\documentclass{amsart}
\usepackage{amsmath}
\usepackage{amssymb} 
\usepackage{amscd} 
\usepackage[dvipdfmx]{graphicx} 
\usepackage{epsfig}
\usepackage[dvips]{color}
\usepackage[all]{xy}  

\newtheorem{theorem}{Theorem}[section]
\newtheorem{lemma}[theorem]{Lemma}

\newtheorem{proposition}[theorem]{Proposition}
\newtheorem{corollary}[theorem]{Corollary}
\newtheorem{claim}[theorem]{Claim}

\newtheorem{conjecture}[theorem]{Conjecture}
\newtheorem{remark}[theorem]{Remark}

\theoremstyle{definition}

\newtheorem*{thm:twist_T}{Theorem~\ref{twist_T}}
\newtheorem*{cor:g-torsion_large_genus}{Corollary~\ref{g-torsion_large_genus}}

\newtheorem*{cor:pretzel}{Corollary~\ref{pretzelknot}}

\numberwithin{equation}{section}
\numberwithin{figure}{section}
\numberwithin{table}{section}

\renewcommand{\(}{\textup{(}}
\renewcommand{\)}{\textup{)}}

\begin{document}
\baselineskip 12pt

\title{Generalized torsion for knots with arbitrarily high genus}

\author[K. Motegi]{Kimihiko Motegi}
\address{Department of Mathematics, Nihon University, 
3-25-40 Sakurajosui, Setagaya-ku, 
Tokyo 156--8550, Japan}
\email{motegi.kimihiko@nihon-u.ac.jp}

\author[M. Teragaito]{Masakazu Teragaito}
\address{Department of Mathematics and Mathematics Education, Hiroshima University, 
1-1-1 Kagamiyama, Higashi-Hiroshima, 739--8524, Japan}
\email{teragai@hiroshima-u.ac.jp}

\dedicatory{Dedicated to the memory of Toshie Takata}

\begin{abstract}
Let $G$ be a group and let  $g$ be a non-trivial element in $G$. 
If some non-empty finite product of conjugates of $g$ equals to the identity,
then $g$ is called a \textit{generalized torsion element}. 
We say that a knot $K$ has generalized torsion if $G(K) = \pi_1(S^3 - K)$ admits such an element. 
For  a $(2, 2q+1)$--torus knot $K$, 
we demonstrate that there are infinitely many unknots $c_n$ in $S^3$ such that  $p$--twisting $K$
about $c_n$ yields a twist family $\{ K_{q, n, p}\}_{p \in \mathbb{Z}}$ in which $K_{q, n, p}$ is a hyperbolic knot 
with generalized torsion whenever $|p| > 3$. 
This gives a new infinite class of  hyperbolic knots having generalized torsion. 
In particular, each class contains knots with arbitrarily high genus. 
We also show that some twisted torus knots, including the $(-2, 3, 7)$--pretzel knot, 
have generalized torsion.  
Since generalized torsion is an obstruction for having bi-order, 
these knots have non-bi-orderable knot groups. 
\end{abstract}

\maketitle

{
\renewcommand{\thefootnote}{}
\footnotetext{2020 \textit{Mathematics Subject Classification.}
Primary 57K10, 57M05, 57M50
\footnotetext{ \textit{Key words and phrases.}
fundamental group, Dehn filling, slope, generalized torsion}
}

\section{Introduction}
\label{Introduction}
Let $G$ be a group and let $g$ be a non-trivial element in $G$. 
If some non-empty finite product of conjugates of $g$ equals to the identity, 
then $g$ is called a \textit{generalized torsion element}.
In particular, any non-trivial torsion element is a generalized torsion element. 

A group $G$ is said to be \textit{bi-orderable\/} if $G$ admits
a strict total ordering $<$ which is invariant under multiplication from
the left and right.
That is, if $g<h$, then $agb<ahb$ for any $g,h,a,b\in G$.
In this paper, the trivial group $\{1\}$ is considered to be bi-orderable.

It is easy to see that a bi-orderable group does not have a generalized torsion element. 
Thus the existence of generalized torsion element is an obstruction for 
a group to be bi-orderable.
It is known that the converse does not hold in general \cite[Chapter 4]{MR}. 
However, among $3$--manifold groups (fundamental groups of $3$--manifolds), 
one may expect that the converse does hold, 
and the authors proposed the following conjecture 
\cite{MT_generalized_torsion}. 

\begin{conjecture}
\label{conj:bo}
Let $G$ be the fundamental group of a $3$--manifold. 
Then,
$G$ is bi-orderable if and only if
$G$ has no generalized torsion element.
\end{conjecture}

Let us restrict our attention to a knot group $G(K)$, 
the fundamental group of the exterior $E(K)$ of a knot $K$ in $S^3$. 
Then  $G(K)$ is known to be left-orderable, 
i.e. it has a a strict total ordering $<$ which is invariant under multiplication from
the left \cite{HW,BRW}. 
On the other hand, 
little is known for having generalized torsion elements or being bi-orderable. 

In what follows, for short, 
we often say that $K$ has generalized torsion if its knot group $G(K)$ has a generalized torsion element. 
Any nontrivial torus knot has generalized torsion \cite{NR}, 
and hence any satellite of a nontrivial torus knot also has generalized torsion. 

Using this simple fact, 
we may observe the following in ad hoc fashion. 

\begin{proposition}
\label{satellite}
For a given knot $K$, 
there are infinitely many twisting circles $c$ such that $K_p$, 
the knot obtained from $K$ by $p$ twisting about $c$,  
has generalized torsion. 
\end{proposition}

\begin{proof}
Let us take an unknotted solid torus $V$ which contains $K$ in its interior so that 
$K$ is not a core of $V$ and not embedded in a $3$--ball in $V$.  
Note that there are infinitely many such solid tori. 
Then let $c$ be an unknotting circle on $\partial V$ which wraps $m\ (\ge 2)$ times in meridional direction and once in the longitudinal direction of $V$.  
If we perform $p$--twisting about $c$, 
then the core of $V$ becomes a $(m, pm+1)$--torus knot. 
Thus $K_p$ is a satellite knot which has the torus knot $T_{m, pm+1}$ as a companion knot. 
Recall that $\pi_1(E(T_{m, pm+1}))$ has a generalized torsion element; see \cite{NR} (cf. \cite{MT_generalized_torsion}). 
Since $\pi_1(E(T_{m, pm+1}))$ injects into $\pi_1(E(K_p))$, 
$\pi_1(E(K_p))$ also has a generalized torsion. 
\end{proof}

By construction, the knots given in Proposition~\ref{satellite} are satellite knots. 
Turning to hyperbolic knots,  
Naylor and Rolfsen \cite{NR} discovered a generalized torsion element 
in the knot group of the hyperbolic knot $5_2$ by using a computer.
It is surprising that this is the first example of hyperbolic knot with generalized torsion. 
The knot $5_2$ is the  $(-2)$-twist knot, 
and  the second named author extends this example to all negative twist knots \cite{Te} . 
As far as we know, 
these twist knots are the only known hyperbolic knots with generalized torsion.   
We emphasize that all twist knots have genus one. 

In this article, applying twisting operation, we demonstrate: 

\begin{theorem}
\label{twist_T}
Let $K_q$ be a $(2, 2q+1)$--torus knot $T_{2, 2q+1}$ \($q \ge 1$\). 
Then there are infinitely many unknots $c_n$ $(n \ge 1)$ in $S^3$ disjoint from $K_q$ such that 
each $c_n$ enjoys the following property. 
Let $K_{q, n, p}$ be a knot obtained from $K_q$ by $p$--twisting about $c_n$. 
Then for each infinite family $\{ K_{q, n, p}\}_{p \in \mathbb{Z}}$,  
$K_{q, n, p}$ is a hyperbolic knot with generalized torsion whenever $|p| > 3$. 
\end{theorem}

Since the linking number between $c_n$ and $K_q$ is greater than one, 
\cite[Theorem~2.1]{BM} shows that the genus of $K_{q, n, p}$ tends to $\infty$ as $|p| \to \infty$.

\begin{corollary}
\label{g-torsion_large_genus}
There are infinitely many hyperbolic knots with arbitrarily high genus, each of which has generalized torsion. 
\end{corollary}

Furthermore, 
we will show that some twisted torus knots, 
including the $(-2, 3, 7)$--pretzel knot, 
have generalized torsion. 
This implies the following.

\begin{corollary}
\label{pretzelknot}
The knot group of the pretzel knot of type $(-2,3,2s+5)$ is not bi-orderable for $s\ge 0$.
\end{corollary}

\medskip
\section{Decomposition of commutators}

We prepare a few useful facts which will be exploited to identify a generalized torsion element. 
Throughout this paper, 
$[x, y]$ denotes the commutator $x^{-1}y^{-1}xy$, and $x^g = g^{-1}x g$ in a group $G$.

Recall the well-known commutator identity which holds in a group: 

\begin{lemma}
\label{identity}
$[x, yz] = [x, z] [x, y]^z$.
\end{lemma}

\begin{proof}
$[x, yz] 
= x^{-1}(yz)^{-1}x (yz) 
=(x^{-1}z^{-1}xz)z^{-1}(x^{-1}y^{-1}xy)z
= [x, z][x, y]^z$.
\end{proof}

\medskip

\begin{lemma}
\label{w}
In a group, let $w(a^{\varepsilon_a}, b^{\varepsilon_b})$ $(\varepsilon_a$ and $\varepsilon_b$ are either $1$ or $-1)$
be any word in which only $a^{\varepsilon_a}$ and $b^{\varepsilon_b}$ appear\textup{;} 
neither $a^{-\varepsilon_a}$ nor $b^{-\varepsilon_b}$ appears. 
Then 
the commutator $[x,w(a^{\varepsilon_a}, b^{\varepsilon_b})]$ can be decomposed into a product of conjugates of $[x, a^{\varepsilon_a}]$ and $[x, b^{\varepsilon_b}]$. 
\end{lemma}

\begin{proof}
The proof is done by the induction on the length of $w(a^{\varepsilon_a}, b^{\varepsilon_b})$. 
For simplicity, we assume $\varepsilon_a = \varepsilon_b = 1$. 
The other cases are similar.

If the length of $w(a, b)$ is one, 
then $w(a, b) = a$ or $b$ by the assumption and $[x, w(a, b)]$ is nothing but $[x, a]$ or $[x, b]$. 

Assume that for any word $w'(a, b)$ with length $n-1$, 
$[x, w'(a, b)]$ can be written as a product of conjugates of $[x, a]$ and $[x, b]$. 
Then we show that the same is true for $[x, w(a, b)]$ for $w(a, b)$ with length $n$. 
Here we suppose that the initial letter of $w(a, b)$ is $a$, i.e. $w(a, b) = aw'(a, b)$. 
Then $[x, w(a, b)] = [x, aw'(a, b)]$. 
Applying Lemma~\ref{identity}, 
we have $[x, aw'(a, b)] = [x, w'(a, b)][x, a]^{w'(a, b)}$. 
Since $w'(a, b)$ has length $n-1$, 
we may write $[x, w'(a, b)]$ as a product of conjugates of $[x, a]$ and $[x, b]$, completing the proof. 
\end{proof}

As a  particular case where $b = x$, 
since $[x, x] = [x, x^{-1}] = 1$, we have: 

\begin{proposition}
\label{w_x}
The commutator
$[x, w(a^{\varepsilon_a}, x^{\varepsilon_x})]$ can be decomposed into a product of conjugates of $[x, a^{\varepsilon_a}]$. 
\end{proposition}

\medskip

\section{Generalized torsion which arises from twisting torus knots $T_{2, 2q+1}$}
\label{trefoil}

The goal of this section is to prove: 

\begin{thm:twist_T}
Let $K_q$ be a $(2, 2q+1)$--torus knot $T_{2, 2q+1}$ \($q \ge 1$\). 
Then there are infinitely many unknots $c_n$ $(n \ge 1)$ in $S^3$ disjoint from $K_q$ such that 
each $c_n$ enjoys the following property. 
Let $K_{q, n, p}$ be a knot obtained from $K_q$ by $p$--twisting about $c_n$. 
Then for each infinite family $\{ K_{q, n, p}\}_{p \in \mathbb{Z}}$,  
$K_{q, n, p}$ is a hyperbolic knot with generalized torsion whenever $|p| > 3$. 
\end{thm:twist_T}

\begin{proof}
For integers $q,\ n \ge 1$ we consider the braid 
\[
(\sigma_1\sigma_2\cdots \sigma_{2q+n+1})(\sigma_1\sigma_2\cdots \sigma_{2q})
\]
of $2q+n+2$ strands,
where $\sigma_i$'s  are the standard generators of the braid group.
Figure \ref{fig:braid} shows this braid and its axis $c_n$.
Let $K_q$ be its closure.
Then, $K_q \cup c_n$ is deformed as shown in Figures \ref{fig:braid1}, \ref{fig:braid2}, \ref{fig:braid3} and \ref{fig:link}, 
where full twists are right-handed.
Hence $K_q$ is the torus knot $T_{2, 2q+1}$. 

Then as shown in Figures~\ref{fig:braid1}, \ref{fig:braid2}, \ref{fig:braid3}, \ref{fig:link},
the link $K_q \cup c_n$ is deformed into a link given in Figure \ref{fig:link2}.

\begin{figure}[h]
\includegraphics*[scale=0.8]{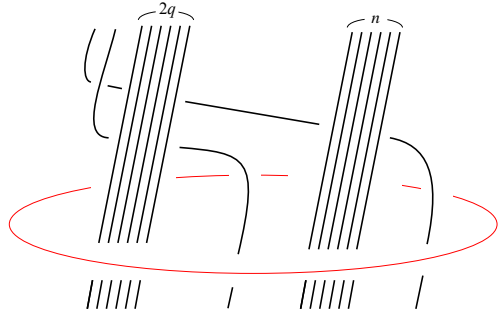}
\caption{A braid with an axis $c_n$.}
\label{fig:braid}
\end{figure}

\begin{figure}[h]
\includegraphics*[scale=0.6]{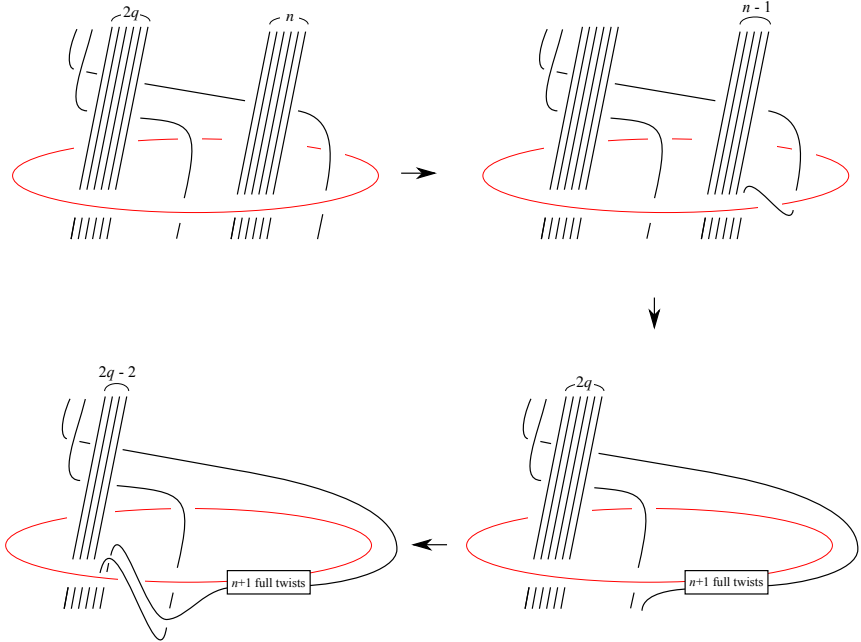}
\caption{Deform $K_q \cup c_n$.}
\label{fig:braid1}
\end{figure}

\begin{figure}[h]
\includegraphics*[scale=0.6]{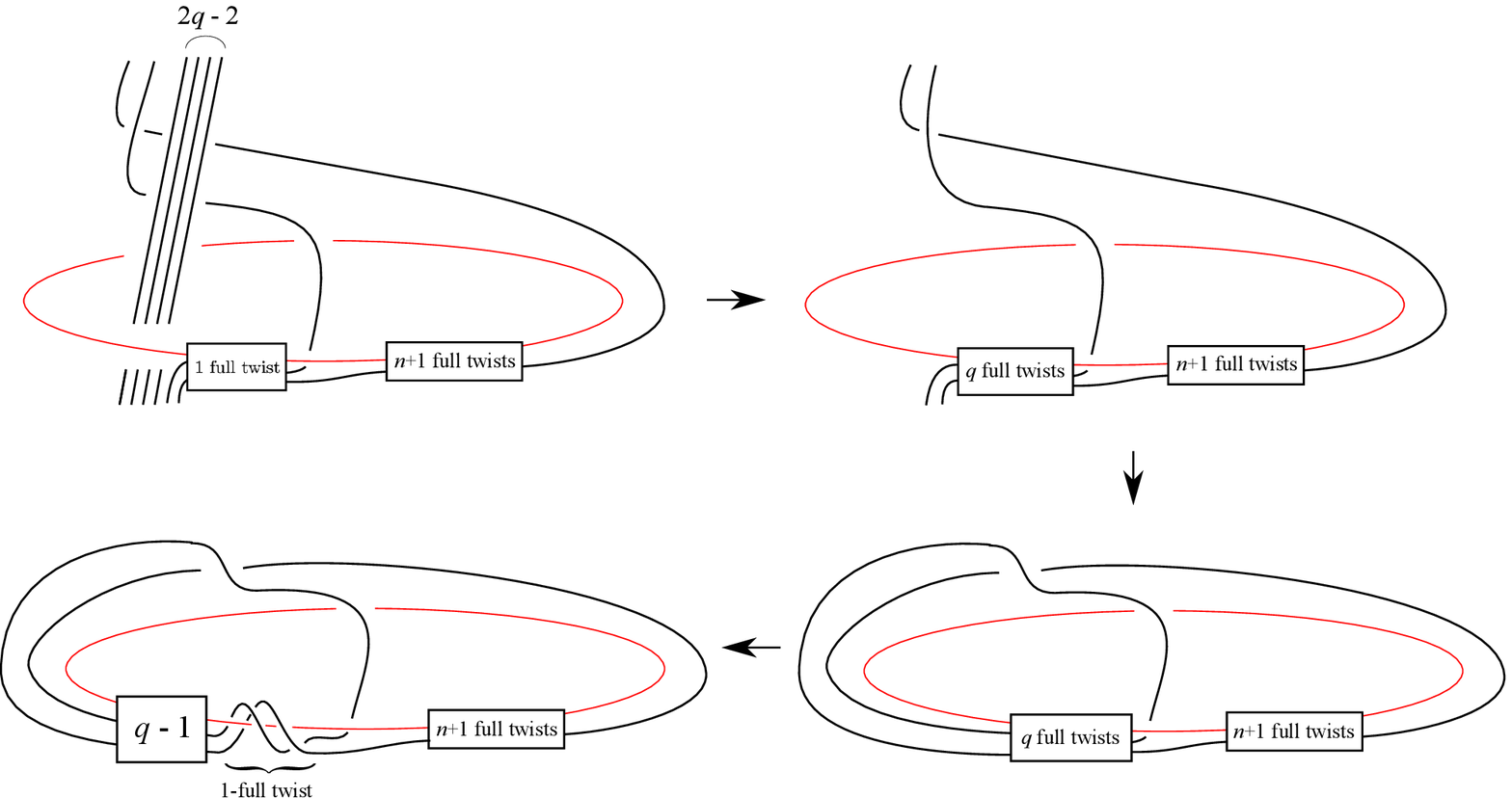}
\caption{Deform $K_q \cup c_n$ (continued).}
\label{fig:braid2}
\end{figure}

\begin{figure}[h]
\includegraphics*[scale=0.6]{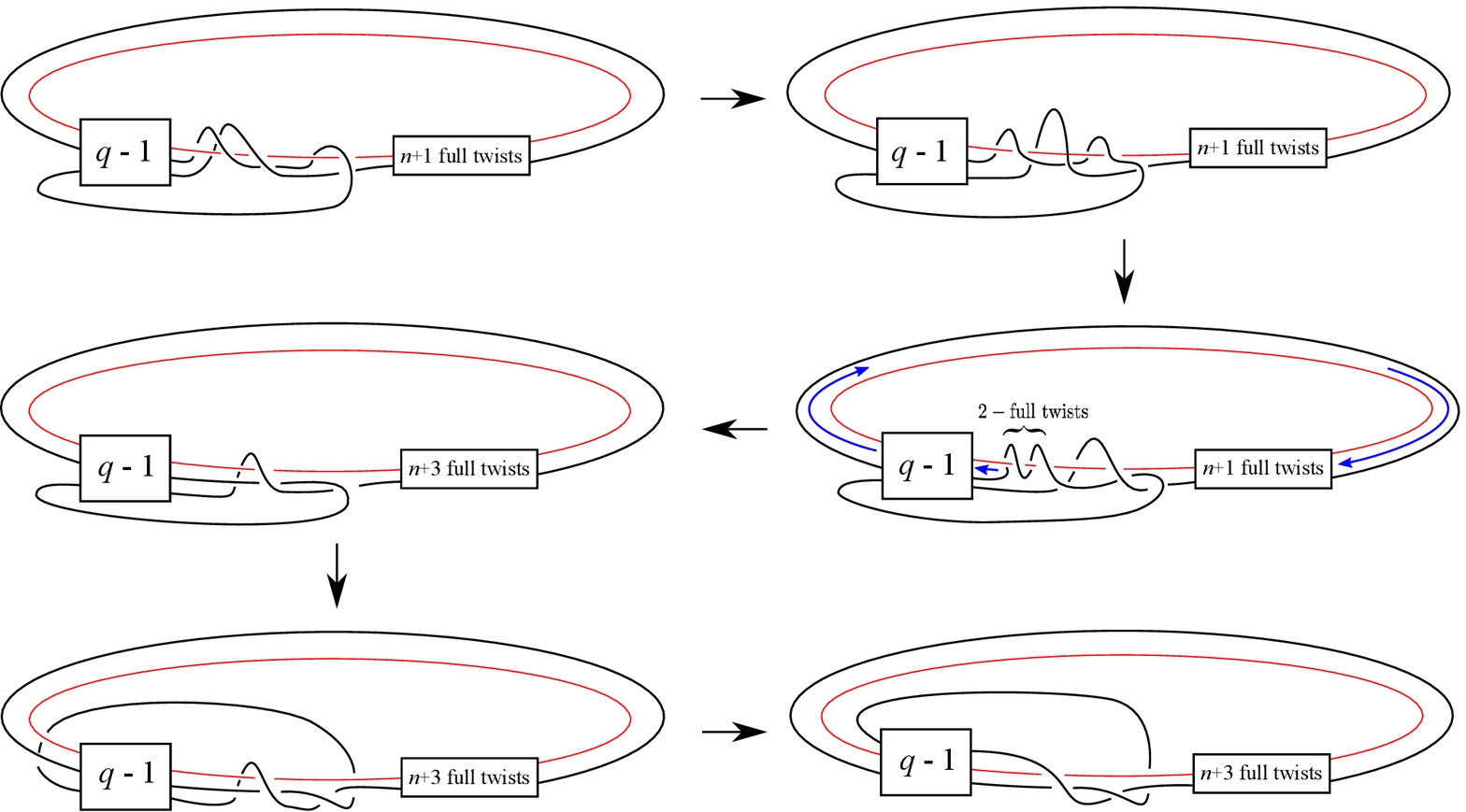}
\caption{Deform $K_q \cup c_n$ (continued).}
\label{fig:braid3}
\end{figure}

\begin{figure}[h]
\includegraphics*[scale=0.4]{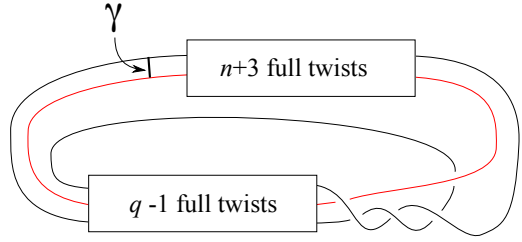}
\caption{The link $K_q \cup c_n$ with unknotting tunnel $\gamma$.}
\label{fig:link}
\end{figure}

\begin{figure}[h]
\includegraphics*[scale=0.4]{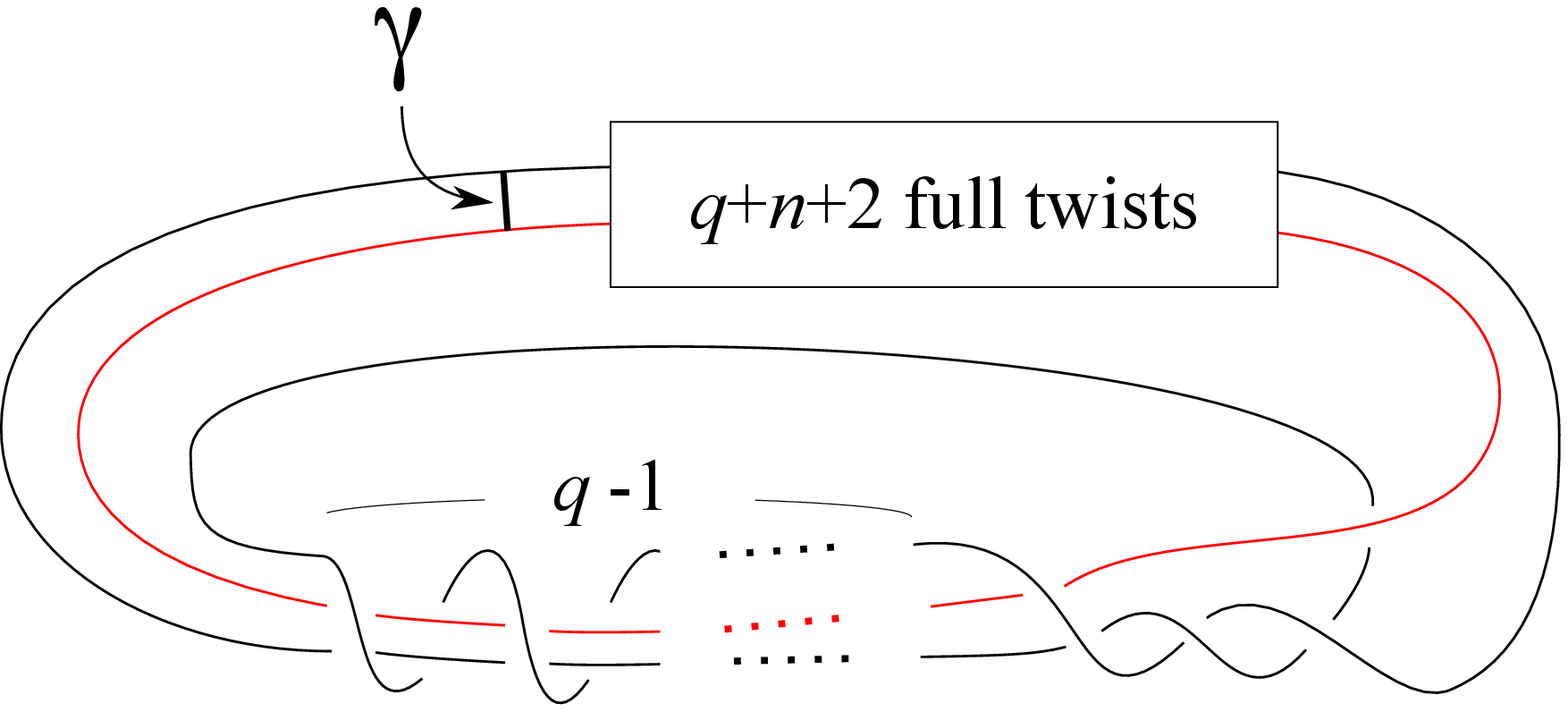}
\caption{$K_q \cup c_n$.}
\label{fig:link2}
\end{figure}

We first claim that the link $K_q \cup c_n$ is a hyperbolic link whose link group
contains a generalized torsion element.

\begin{claim}
\label{cl:hyp}
$K_q \cup c_n$ is hyperbolic.
\end{claim}

\begin{proof}[Proof of Claim \ref{cl:hyp}]
The braid given by Figure~\ref{fig:braid} is the mirror image of the braid 
in \cite[Figure 18 (c)]{HK}. 
So it is pseudo-Anosov by \cite[Theorem 3.11]{HK}. 
(In \cite{KR}, Theorem 6.7 treats the case where $n=1$.
Also, in Figure 18(b) and (c) of \cite{HK},
there is a mistake.
The number  $2m-1$ of strands should be $2m$. )
Hence the link complement of $K_q \cup c_n$ is hyperbolic.
\end{proof}

\begin{claim}
\label{cl:g}
The link group of $K_q \cup c_n$ contains a generalized torsion element.
\end{claim}

\begin{proof}[Proof of Claim \ref{cl:g}]
The link $K_q \cup c_n$ has tunnel number one.
Let $\gamma$ be its unknotting tunnel as shown in Figure \ref{fig:link}.
This means that the outside of the regular neighborhood $N(K_q \cup c_n \cup \gamma)$ 
is a genus two handlebody.
Let $\ell$ be the co-core loop of $N(\gamma)\subset N(K_q \cup c_n \cup \gamma)$.
We deform $N=N(c_n \cup \gamma)$ with $\ell$ as shown in Figure \ref{fig:deform1}, 
where $q+n+2$ full twists are expressed as $-1/(q+n+2)$--surgery along an unknotted circle.
Figures \ref{fig:deform3}, \ref{fig:deform4}, and \ref{fig:deform5}
show the deformation of $N$ with $\ell$, 
where $\ell$ is expressed as a band sum of two circles on $\partial N$ as in Figure~\ref{fig:deform1}. 
Then subsequently deform $N$ as shown in Figures~\ref{fig:deform3} and \ref{fig:deform4}.
Finally, the $q$ twists in Figure \ref{fig:deform4} is absorbed as illustrated in Figure \ref{fig:deform6}.

\begin{figure}[h]
\includegraphics*[scale=0.4]{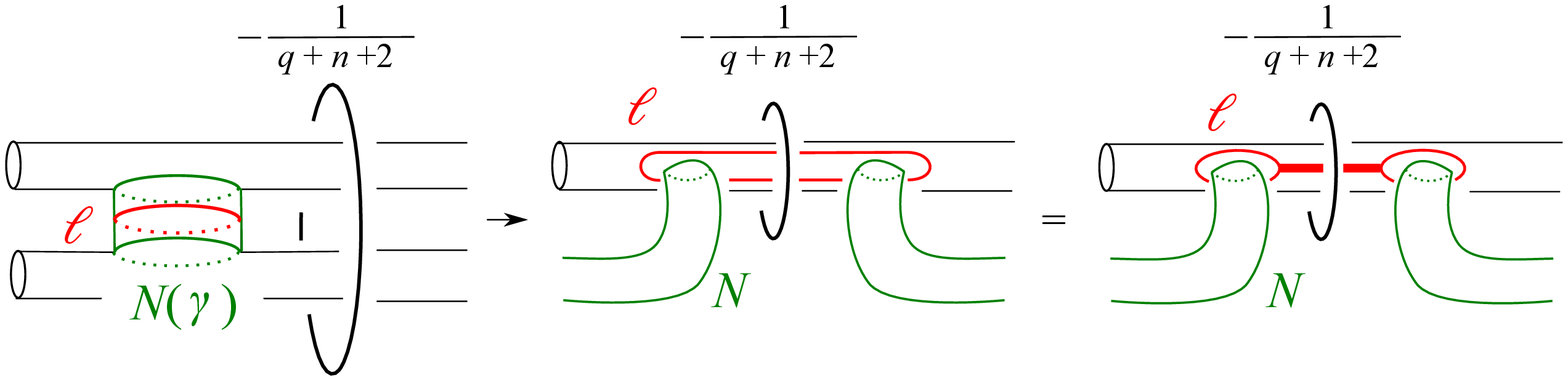}
\caption{Deform $N$ with $\ell$.}
\label{fig:deform1}
\end{figure}

\begin{figure}[h]
\includegraphics*[scale=0.5]{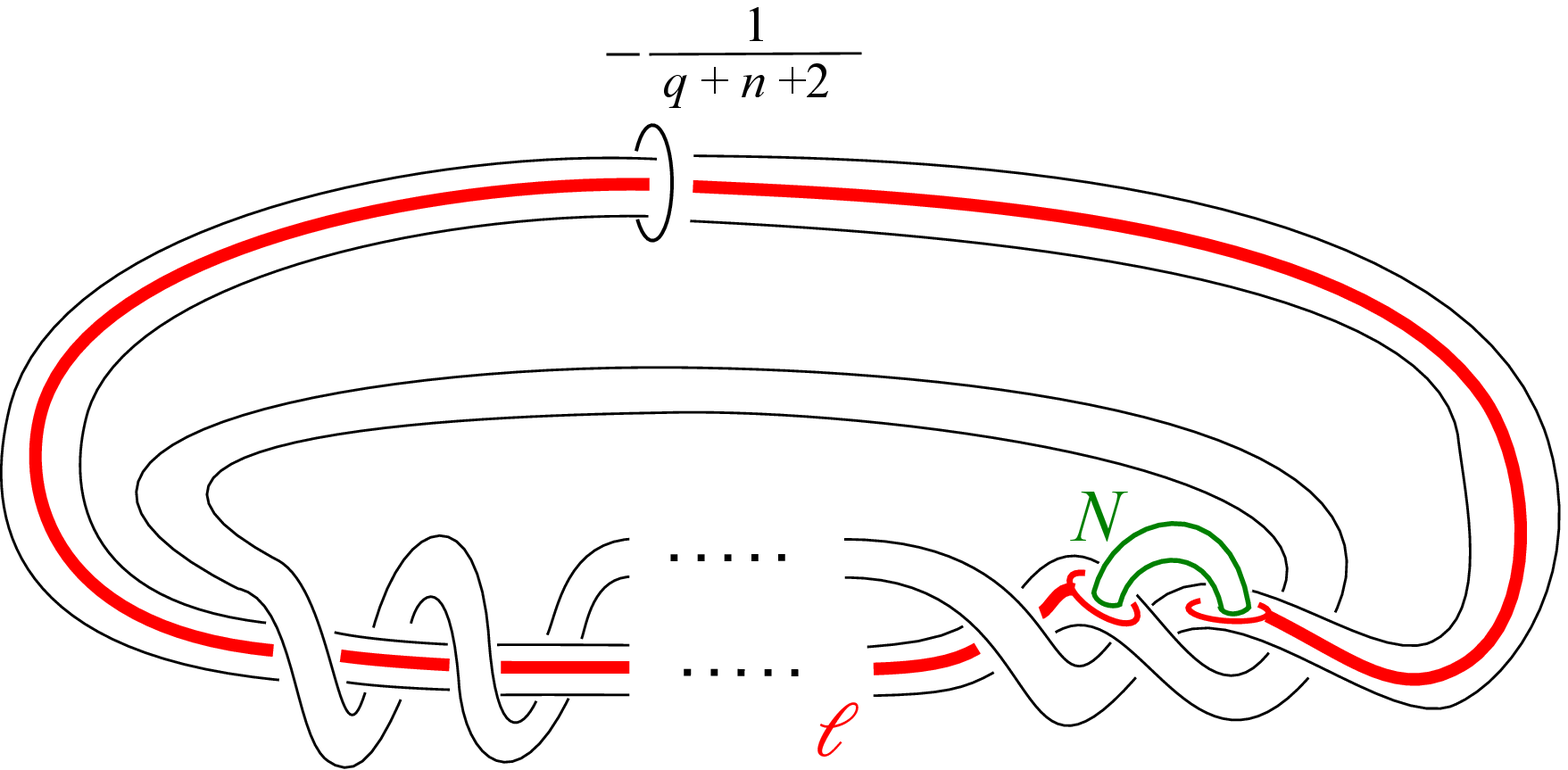}
\caption{Deform $N$ with $\ell$ (continued).}
\label{fig:deform3}
\end{figure}

\begin{figure}[h]
\includegraphics*[scale=0.5]{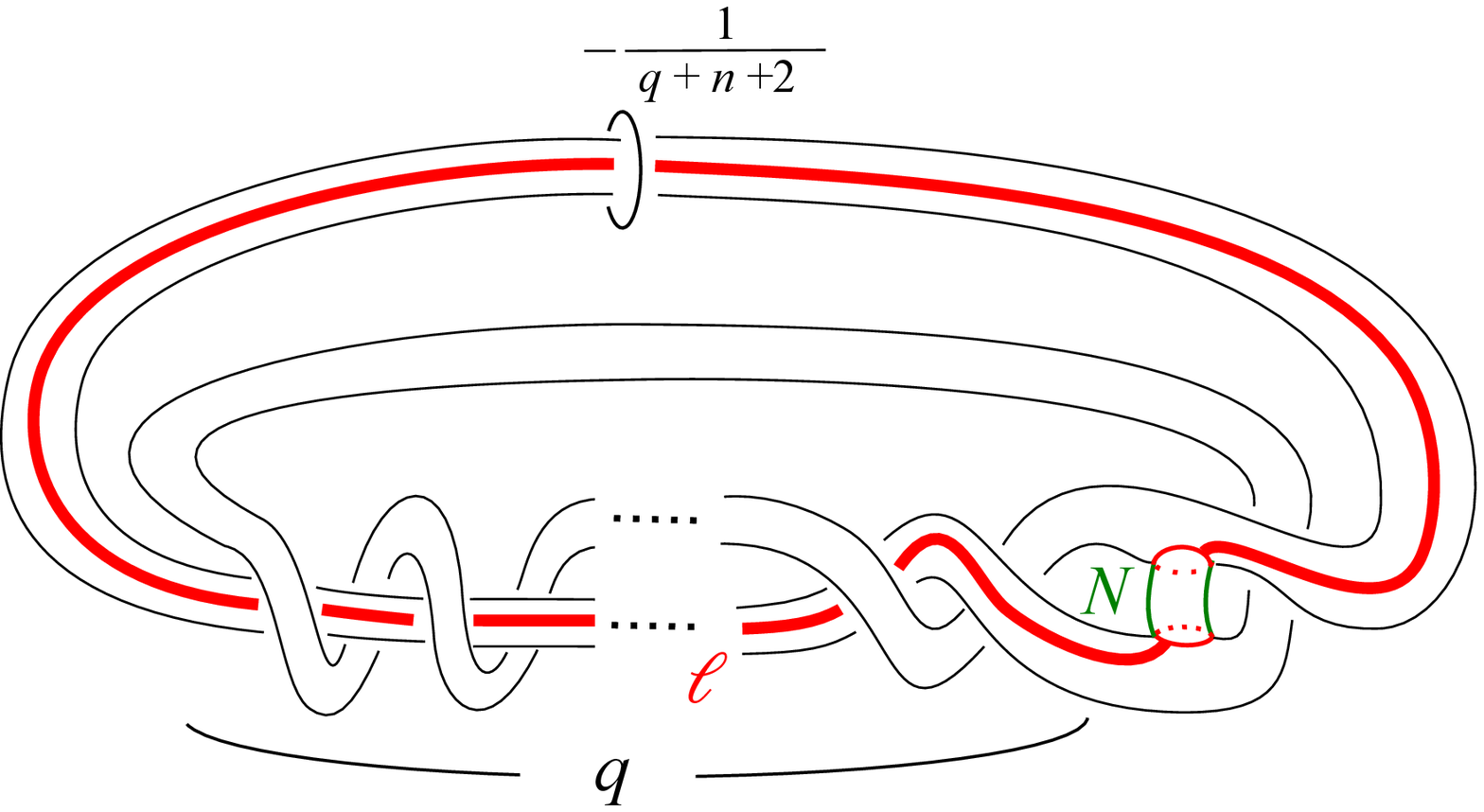}
\caption{Deform $N$ with $\ell$ (continued).}
\label{fig:deform4}
\end{figure}

\begin{figure}[h]
\includegraphics*[scale=0.4]{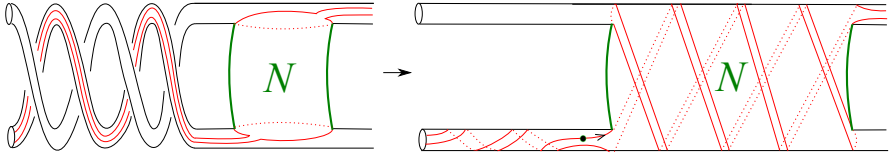}
\caption{Absorb $q$ twists.}
\label{fig:deform6}
\end{figure}

\begin{figure}[h]
\includegraphics*[scale=0.4]{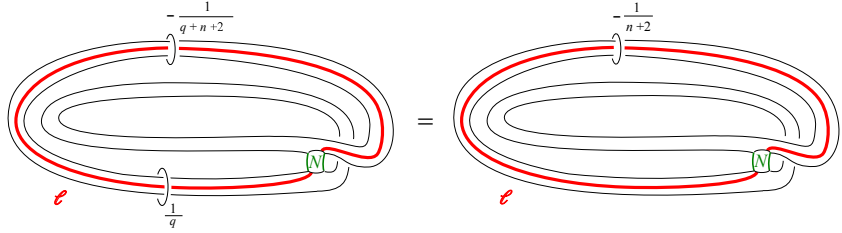}
\caption{Deform $N$ with $\ell$ (continued); $\ell$ is twisted in $N$ as in the right-hand side of Figure~\ref{fig:deform6}.}
\label{fig:deform5}
\end{figure}


\begin{figure}[h]
\includegraphics*[scale=0.5]{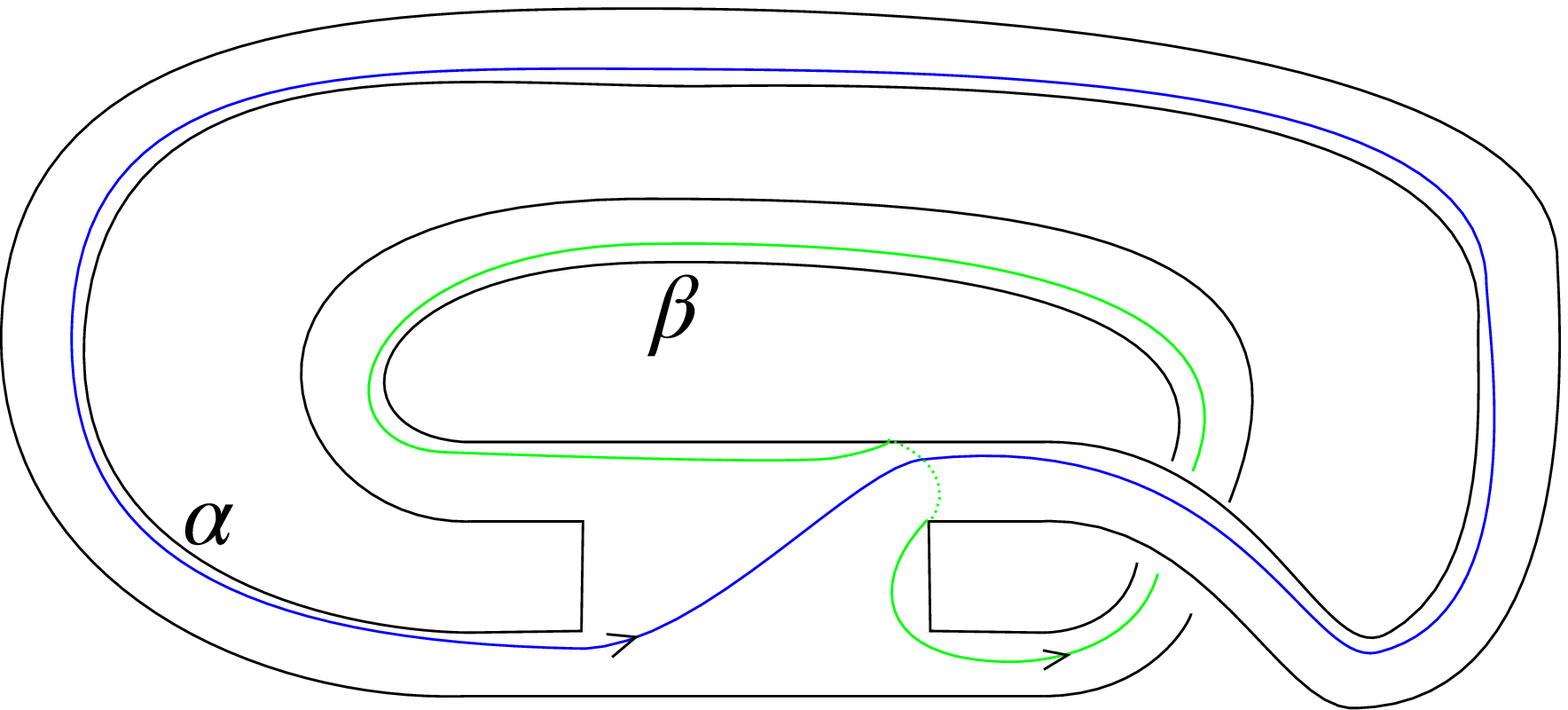}
\caption{Two loops $\alpha$ and $\beta$ bound mutually disjoint meridian disks of the outside handlebody $H$.}
\label{fig:deform7}
\end{figure}

In the final form of Figure \ref{fig:deform5}, 
it is obvious to see that the outside of $N(K_q \cup c_n \cup \gamma)$ is a genus two handlebody denoted by $H$.
As in Figure \ref{fig:deform7},
the loops $\alpha$ and $\beta$, which lie on $\partial N$, 
bound mutually disjoint non-separating meridian disks of $H$.
If we take generators $a$ and $b$ of $\pi_1(H)$ as duals of $\alpha$ and $\beta$,
then $\pi_1(H)$ is a rank two free group generated by $a$ and $b$.
The link exterior of $K_q \cup c_n$ is obtained from $H$ by attaching a $2$--handle along $\ell$.
By following the intersection points of $\ell$ with $\alpha$ and $\beta$,
we can represent $\ell$ as a word $w$ of $a$ and $b$.
In fact, by choosing the base point and an orientation for $\ell$ as shown in Figures \ref{fig:deform6}. 
Then referring Figure~\ref{fig:deform7},
we obtain 
\begin{eqnarray*}
w_{q, n} 
&=& a(b^{-1}a^{-1})^{q}b^{-1}(ab)^{q}a^{n+2}b(ab)^{q}(a^{-1}b^{-1})^{q}a^{-1}a^{-(n+2)} \\
&=& a(b^{-1}a^{-1})^{q}b^{-1}(ab)^{q}a^{n+2}(ba)^{q}b(a^{-1}b^{-1})^{q}a^{-(n+3)}
\end{eqnarray*}
Hence we have
\begin{align*}
\pi_1(S^3-K_q \cup c_n)
&=\langle a, b \mid w_{q, n} = 1 \rangle\\
&=\langle a, b \mid    
(ab)^{q}a^{n+2}(ba)^{q}b=b(ab)^{q}a^{n+2}(ba)^{q} \rangle\\
&=\langle a, b \mid [b, (ab)^{q}a^{n+2}(ba)^{q}]=1 \rangle.
\end{align*}

Proposition~\ref{w_x} shows that 
the commutator $[b, (ab)^{q}a^{n+2}(ba)^{q}]$ 
is decomposed into a product of conjugates of $[b, a]$.  
Hence $[b, a]$ is a generalized torsion element in $\pi_1(S^3-K_q \cup c_n)$ 
if it is nontrivial. 
Assume for a contradiction that $[b, a] = 1$ in $\pi_1(S^3-K_q \cup c_n)$. 
Then, since $a$ and $b$ generate $\pi_1(S^3-K_q \cup c_n)$, 
$\pi_1(S^3-K_q \cup c_n)$ would be abelian. 
However, 
the unknot and the Hopf link are the only knot and link with abelian knot or link group, 
contradicting Claim~\ref{cl:hyp}. 
Thus $[b,a]$ is nontrivial, 
and hence a generalized torsion element in $\pi_1(S^3-K_q \cup c_n)$.
\end{proof}

\begin{claim}
\label{g-torsion_Kqnp}
$[b, a]$ is a generalized torsion element of $G(K_{q, n, p})$, provided if $|p| \ne 1$. 
\end{claim}

\begin{proof}
We remark that $E(K_{q, n, p})$ is obtained from $S^3 - \mathrm{int}N(K_q \cup c_n)$ by 
$(-\frac{1}{p})$--Dehn filling along $\partial N(c_n)$. 
This gives us an epimorphism
\[
\varphi \colon \pi_1(S^3 - \mathrm{int}N(K_q \cup c_n)) \to G(K_{q, n, p}) = \pi_1(E(K_{q, n, p})). 
\]
Recall that $a$ and $b$ generate $\pi_1(S^3 - \mathrm{int}N(K_q \cup c_n))$ and satisfy 
$[b, (ab)^q a^{n+2}(ba)^q] = 1$. 
For notational simplicity, in the following we use the same symbol $a$ and $b$ to denote 
$\varphi(a)$ and $\varphi(b)$. 
Then $G(K_{q, n, p})$ is generated by $a$ and $b$. 
Since $[b, (ab)^q a^{n+2}(ba)^q] = 1 \in G(K_{q, n, p})$, 
by Proposition~\ref{w_x}, 
$[b, a]$ is a generalized torsion element of $G(K_{q, n, p})$ whenever it is nontrivial. 
 
Let us assume that $[b, a]$ is trivial in $G(K_{q, n, p})$. 
Then $G(K_{q, n, p})$ is abelian, and hence $G(K_{q, n, p}) \cong \mathbb{Z}$, 
i.e. $K_{q, n, p}$ is a trivial knot. 
If $n = 0$, then $K_{q, n, 0} = K_{q, n}$ is a nontrivial torus knot $T_{2, 2q+1}$.
Hence $n \ne 0$.  
Note that the torus knot space $E(K_{q, n}) = E(K_{q, n, 0}) = E(T_{2, 2q+1})$ is obtained from 
the solid torus $V = S^3 - \mathrm{int}N(K_{q, n, p})$ by $\frac{1}{p}$--surgery on $c_n$. 
Since $E(K_{q, n})$ is a Seifert fiber space and $V- \mathrm{int}N(c_n)$ is hyperbolic (Claim~\ref{cl:hyp}), 
we may apply \cite[Theorem~1.2]{MM3} to conclude that $|p| = 1$. 
\end{proof}

\begin{remark}
In Claim \ref{g-torsion_Kqnp},
$K_{q, n, \pm 1}$ also has a generalized torsion. 
We may observe that $K_{q, n, 1}$ is a closure of a positive braid, 
and $K_{q, n, -1}$ is a closure of a negative braid.
Hence \cite{Stallings} shows that both have a positive genus. 
This implies that $K_{q,n,\pm 1}$ is nontrivial. 
\end{remark}

\begin{claim}
\label{Kqnp_hyp}
$K_{q, n, p}$ is a hyperbolic knot if $|p| > 3$. 
\end{claim}

\begin{proof}
Since $K_{q, n}$ is a nontrivial torus knot $T_{2, 2q+1}$ and $S^3 - \mathrm{int}N(K_q \cup c_n)$ is hyperbolic (Claim~\ref{cl:hyp}), 
\cite[Proposition~5.11]{DMM} shows that $K_{q, n, p}$ is hyperbolic if $|p| > 3$. 
\end{proof}

\begin{claim}
\label{infinite_c_n}
$K_q \cup c_n$ and $K_q \cup c_{n'}$ are not isotopic when $n \ne n'$ 
\end{claim}

\begin{proof}
As seen in Figure \ref{fig:braid},
the linking number between $K_q$ and $c_n$ is $2q+n+2$ with a suitable orientation.
Hence, if $n \ne n'$, 
$K_q \cup c_n$ and $K_q \cup c_{n'}$ are not isotopic. 
\end{proof}

Thus we obtain infinitely many twisting circles $c_n$ for $K _q = T_{2, 2q+1}$ by varying $n$. 

\medskip

Now the proof of Theorem~\ref{twist_T} follows from Claims~\ref{g-torsion_Kqnp}, \ref{Kqnp_hyp} and 
\ref{infinite_c_n}. 
\end{proof}

\medskip

\begin{remark}
\begin{enumerate}
\item
When $q=0$,
the link $K_0\cup c_n$ is equivalent to the pretzel link of type $(-2,3,2n+6)$.
This special case is treated in  \cite{Tera_link}.
\item
The generalized torsion element for $K_{q,n,p}$ is derived from that of the link $K_q\cup c_n$.
Hence, the element lies in the complement of $K_q\cup c_n$, and furthermore, the $2$-complex which realizes
the triviality of the product of its conjugates also lies there.
The word representing the element is $[b,a]$ in common, but  the generators $a$ and $b$ of $\pi_1(S^3-\mathrm{int} N(K_q\cup c_n))$ 
depend on the parameters $q$ and $n$.

\end{enumerate}
\end{remark}

As an application we have the following. 

\begin{cor:g-torsion_large_genus}
There are infinitely many hyperbolic knots with arbitrarily high genus, each of which has a generalized torsion.
\end{cor:g-torsion_large_genus}

\begin{proof}
Let us take the twist family $\{ K_{q, n, p} \}_{p \in \mathbb{Z}}$ given in Theorem~\ref{twist_T}. 
Then $K_{q, n, p}$ is a hyperbolic knot with a generalized torsion element whenever $|p| > 3$. 
Since the linking number between $K_q$ and $c_n$ is greater than one, 
\cite[Theorem~2.1]{BM} shows that the genus of $K_{q, n, p}$ tends to $\infty$ as $p \to \infty$.
\end{proof}

The next is a slight generalization of \cite[Theorem 6.7]{KR}.
An $n$-strand braid naturally induces an automorphism of the free group $F_n$ of rank $n$. 
It is well-known that $F_n$ is bi-orderable.  
The braid is said to be \textit{order-preserving} if the corresponding automorphism preserves some bi-ordering of $F_n$.
See \cite{KR} for details.


\begin{corollary}
For an integer $q\ge 1$ and $n\ge 1$,
the braid
\[
(\sigma_1\sigma_2\cdots \sigma_{2q +n+1})(\sigma_1\sigma_2\cdots \sigma_{2q})
\]
is not order-preserving.
\end{corollary}

\begin{proof}
As shown in the proof of Theorem \ref{twist_T} (Claim~\ref{cl:g}), 
the link $K_q \cup c_n$ has a generalized torsion element in its link group, 
so the group is not bi-orderable.
By \cite[Proposition 4.1]{KR},  this is equivalent to the conclusion. 
\end{proof}

%
%
%
%
%
%
%

\medskip

\section{Twisted torus knots}

In this section, we give several families of twisted torus knots, whose knot groups have
generalized torsion elements.

Let $p\ge 2$ and $m, s\ge 1$, and
let $K$ be the twisted torus knot $K(p(m+1)+1,pm+1; 2, s)$. 

\begin{lemma}
The knot group $G(K)$ has a presentation
\begin{equation}
\label{presentation_K(p(m+1)+1,pm+1; 2, s)}
\begin{split}
\langle a, c \mid\  &
a^{(p-1)(m+1)+1}(    a^{-(p-2)(m+1)-1}   c^{(p-2)m+1}  )^s a^{m+1}= \\
&  c^{(p-1)m+1} (  a^{-(p-2)(m+1)-1}   c^{(p-2)m+1}  )^s c^{m}
\rangle
\end{split}
\end{equation}

\end{lemma}

\begin{proof}
We follow the argument of \cite{CGHV,CW}.
Let $\Sigma$ be the standard genus two Heegaard surface of $S^3$ 
with the standard generators $a, b, c, d$ of $\pi_1(\Sigma)$. 
Then it bounds genus two handlebodies $U$ inside and $V$ outside. 
Note that $a, b$ generate $\pi_1(U)$ and $c, d$ generate $\pi_1(V)$. 
For convenience we use the same symbols $a, b$ to denote $i_*(a), i_*(b) \in \pi_1(U)$, 
where $i \colon \Sigma \to U$ is the inclusion, 
and similarly use the same symbols $c,d$ to denote $j_*(c), j_*(d) \in \pi_1(V)$, 
where $j \colon \Sigma \to V$ is the inclusion. 

We put $K_0$ as illustrated in Figure \ref{fig:heegaard}. 
Then $\Sigma-K_0$ retracts to the wedge of three circles $G_0$, $R_0$ and $P_0$.
Hence  $G_0$, $R_0$ and $P_0$ represent generators of $\pi_1(\Sigma-K_0)$. 
We also note that $G_0$, $R_0$ and $P_0$ represent $b$, $d$ and $c$, respectively.  

\begin{figure}[h]
\includegraphics*[scale=0.55]{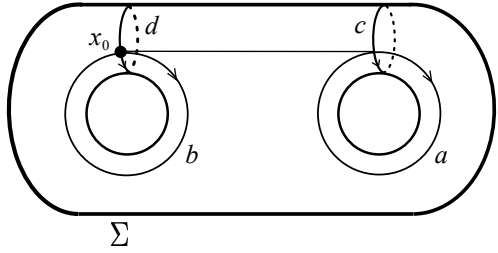}
\caption{The Heegaard surface $\Sigma$ and the standard generators $a, b, c, d$ of $\pi_1(\Sigma)$.}
\label{fig:heegaard}
\end{figure}

Let us take the circles $C_1$, $C_2$, $C_3$ and $C_4$ on $\Sigma$ as shown in Figure \ref{fig:heegaard2} 
so that these curves are disjoint from the base point $x_0$ of the fundamental group $\pi_1(\Sigma)$.

\begin{figure}[h]
\includegraphics*[scale=0.5]{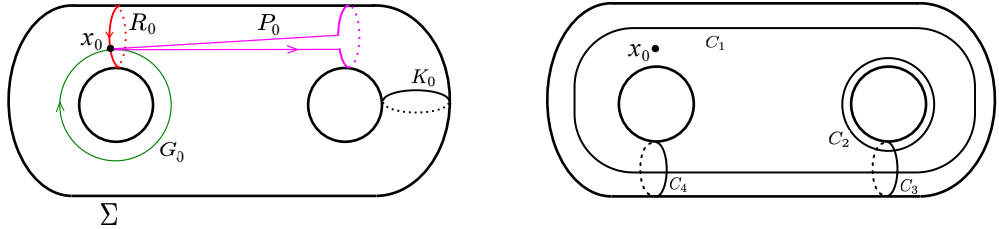}
\caption{$K_0$ and the generators $[G_0]$, $[R_0]$, $[P_0]$ of $\pi_1(\Sigma -K_0)$ (Left); 
the circles $C_1$, $C_2$, $C_3$ and $C_4$ (Right).}
\label{fig:heegaard2}
\end{figure}

Let $\varphi \colon \Sigma \to \Sigma$ be an automorphism which is obtained by 
composing Dehn twists along these curves in the following order:
\begin{itemize}
\item[(D1)]
2 times along $C_1$ to the left.
\item[(D2)]
$p-2$ times along $C_2$ to the left.
\item[(D3)]
$m$ times along $C_3$ to the right.
\item[(D4)]
Once along $C_2$ to the left.
\item[(D5)]
$s$ times along $C_4$ to the right.
\end{itemize}
We may assume that each Dehn twist fixes a base point $x_0 \in \Sigma$, and hence so does $\varphi$. 
Then $K = \varphi(K_0)$ is our twisted torus knot $K(p(m+1)+1,pm+1; 2, s)$; 
we denote $G = \varphi(G_0),\ R = \varphi(R_0)$ and $P = \varphi(P_0)$. 
They represent generators of $\pi_1(\Sigma - K)$. 

To give a presentation of $G(K)$, 
we consider the following decomposition: 
$S^3 - K = (U-K) \cup (V-K)$,\ $(U-K) \cap (V-K) = \Sigma - K$. 
Recall that $\pi_1(U-K) \cong \pi_1(U)$ is generated by $a, b$, 
and $\pi_1(V-K) \cong \pi_1(V)$ is generated by $c, d$. 
We also recall that $\pi_1(\Sigma - K)$ is generated by $[G], [R], [P]$. 

We express $[G]$, $[R]$ and $[P]$ in $\pi_1(\Sigma)$ using the standard generators $a, b, c$ and $d$, 
and then push them into $\pi_1(U) = \pi_1(U-K)$ and $\pi_1(V) = \pi_1(V-K)$.  
To this end, for convenience, 
we collect the effect of the above Dehn twists (D1)--(D5) on generators $a, b, c$ and $d$. 

\begin{table}[h]
\label{effection}
\begin{tabular}{|c|c|c|c|c|}
\hline
         & $a \mapsto$ & $b \mapsto$ & $c \mapsto$ & $d \mapsto$ \\
\hline
$D1$ &  $a$  & $b$ & $c(ab)^2$ & $d(ab)^2$ \\
\hline
$D2$ &  $a$  &  $b$ & $a^{p-2}c$ & $d$  \\
\hline
$D3$  & $ac^m$ & $b$ & $c$ &  $d$   \\
\hline
$D4$  & $a$ & $b$ & $ac$ &  $d$   \\
\hline
$D5$  & $a$ & $d^sb$ & $c$ &  $d$   \\
\hline
\end{tabular}
\end{table}

Recall that $[G] = b,\ [R] = d$ and $[P] = c$. 
Following the table, for $[G] = b$, 
we have:  
\[
[G] = b \mapsto b \mapsto b \mapsto b \mapsto b \mapsto d^sb.
\] 
Thus $[G] = b \in \pi_1(U)$ (putting $d = 1$) 
and $[G] = d^s \in \pi_1(V)$ (putting $b = 1$) . 
Similarly for $[R] = d$, 
we have: 
\[
[R] = d \mapsto d(ab)^2 \mapsto d(ab)^2 \mapsto d\bigl((ac^m)b\bigr)^2 \mapsto d\bigl(a(ac)^mb\bigr)^2 
\mapsto d\bigl(a(ac)^m(d^sb)\bigr)^2.
\] 
Thus $[R] =  a^{m+1}ba^{m+1}b \in \pi_1(U)$ (putting $c = d = 1$) 
and $[R] = dc^md^sc^m d^{s} \in \pi_1(V)$ (putting $a = b = 1$) . 

For $[P] = c$, 
we have: 
\begin{equation*}
\begin{split}
[P] = & c \mapsto c(ab)^2 \mapsto (a^{p-2}c)(ab)^2 \mapsto  (ac^m)^{p-2}c\bigl((ac^m)b\bigr)^2 
\mapsto \bigl(a(ac)^m\bigr)^{p-2}(ac)\bigl(a(ac)^mb\bigr)^2 \\
& \mapsto (a(ac)^m)^{p-2}ac\bigl(a(ac)^m(d^sb)\bigr)^2.
\end{split}
\end{equation*}
Thus $[P] =  a^{(m+1)(p-2)}a(a^{m+1}b)^2 = a^{(p-1)(m+1)+1} b a^{m+1}b \in \pi_1(U)$ (putting $c = d = 1$) 
and $[P] = c^{(p-1)m+1}d^sc^md^s \in \pi_1(V)$ (putting $a = b = 1$) . 

The results are summarized as follows.

\begin{table}[h]
\label{summary}
\begin{tabular}{|c|c|c|}
\hline
 & $\pi_1(U)$  & $\pi_1(V)$ \\
\hline
$[G]$ &   $b$   & $d^s$    \\
\hline
$[R]$ &    $a^{m+1}ba^{m+1}b$  &  $dc^md^sc^m d^{s}$  \\
\hline
$[P]$  & $a^{(p-1)(m+1)+1} b a^{m+1}b$     &    $c^{(p-1)m+1}d^sc^{m}d^s$   \\
\hline
\end{tabular}
\end{table}

By the Seifert-van Kampen Theorem,
$G(K)$ has a presentation
\begin{equation*}
\begin{split}
G=\langle a,b,c,d \mid
& \  b=d^s, 
a^{m+1}ba^{m+1}b=dc^md^sc^m d^{s} \\
&\qquad 
a^{(p-1)(m+1)+1}ba^{m+1}b=c^{(p-1)m+1}d^sc^{m}d^s \rangle.
\end{split}
\end{equation*}
This is equivalent to
\begin{equation*}\label{p1}
\langle a,c,d \mid a^{m+1}d^s a^{m+1}=dc^md^sc^m,
a^{(p-1)(m+1)+1}d^sa^{m+1}=c^{(p-1)m+1}d^sc^m \rangle.
\end{equation*}
The second relation is changed to
\[
a^{(p-2)(m+1)+1 }\cdot a^{m+1}d^s a^{m+1}=c^{(p-1)m+1}d^sc^{m}.
\]
By using the first relation, this gives
\[
a^{(p-2)(m+1)+1}\cdot dc^md^sc^m=c^{(p-1)m+1}d^sc^m.
\]
So, we have
\[
a^{(p-2)(m+1)+1}d=c^{(p-2)m+1}.
\]
By deleting the generator $d$, we have
\[
\begin{split}
\langle a, c \mid\  &
a^{m+1}( a^{-(p-2)(m+1)-1}   c^{(p-2)m+1}  )^s a^{m+1}= \\
&a^{-(p-2)(m+1)-1}  c^{(p-2)m+1} c^m(  a^{-(p-2)(m+1)-1}   c^{(p-2)m+1}  )^s c^{m}
\rangle.
\end{split}
\]
Finally, this is equivalent to
\[
\begin{split}
\langle a, c \mid\  &
a^{(p-1)(m+1)+1}(    a^{-(p-2)(m+1)-1}   c^{(p-2)m+1}  )^s a^{m+1}= \\
&  c^{(p-1)m+1} (  a^{-(p-2)(m+1)-1}   c^{(p-2)m+1}  )^s c^{m}
\rangle
\end{split}
\]
as desired.
\end{proof}

\begin{remark}
\label{genus}
The twisted torus knot $K(p(m+1)+1,pm+1; 2, s)$ \textup{(}$p \ge 2, m, s \ge 1$\textup{)} is the closure of a positive braid 
with braid index $p(m+1)+1$ and word length $(p(m+1)+1 -1)(pm+1) + 2s = p(m+1)(pm+1) + 2s$. 
Hence  the genus of $K(p(m+1)+1,pm+1; 2, s)$ is given by 
$\dfrac{1 - (p(m+1)+1) + p(m+1)(pm+1) + 2s}{2} = \dfrac{p^2m(m+1)}{2} + s$; see \cite{Stallings}. 
\end{remark}

\begin{theorem}\label{thm:pm}
Let $p\ge 2$ and $m\ge 1$.
The knot group $G$ of the twisted torus knot $K(p(m+1)+1,pm+1; 2,1)$ admits
a generalized torsion element.
\end{theorem}

\begin{proof}
Putting $s=1$ in (\ref{presentation_K(p(m+1)+1,pm+1; 2, s)}) and rewriting the relation, 
$G$ has a presentation
\[
\begin{split}
G=\langle a,c \mid\ &
a^{(p-1)(m+1)+1}a^{-(p-2)(m+1)-1}c^{(p-2)m+1}a^{m+1}=\\
& c^{(p-1)m+1}a^{-(p-2)(m+1)-1}c^{(p-2)m+1}c^m
\rangle \\
=
\langle a,c \mid\ &
a^{m+1}c^{(p-2)m+1}a^{m+1}=
 c^{(p-1)m+1}a^{-(p-2)(m+1)-1}c^{(p-1)m+1}
\rangle.
\end{split}
\]
The relation is changed to
\[
c^{(p-2)m+1}=a^{-(m+1)}c^{(p-1)m+1}a^{-(p-2)(m+1)-1}c^{(p-1)m+1}a^{-(m+1)}.
\]


Let $w(a^{-1},c)$ be the right hand side of this relation.
Then the commutator $[c,w(a^{-1},c)]$ is the identity, 
and it is decomposed into a product of conjugates of $[c,a^{-1}]$ (Proposition~\ref{w_x}).

Once we know that $[c,a^{-1}]\ne 1$, this give a generalized torsion element in $G$.
If $[c,a^{-1}]=1$ in $G$, then $G$ would be abelian.
Hence $K$ must be trivial.
However, this is impossible, because $K$ is a closure of a positive braid, so a
fibered knot of genus $p^2m(m+1)/2+1$.
\end{proof}

In Theorem \ref{thm:pm}, for a technical reason, 
we assumed $s = 1$, while we may vary $p$ and $m$. 
On the other hand, if we put $p=2$ and $m=1$, 
then we may vary the twisting parameter $s \ge 0$. 

\begin{theorem}\label{thm:532s}
\label{K(5,3; 2,s)}
The knot group $G$ of the twisted torus knot $K(5,3; 2,s)$ admits
a generalized torsion element for any $s\ge 0$. 
\end{theorem}

\begin{proof}
Applying Tietze transformations to the presentation (\ref{presentation_K(p(m+1)+1,pm+1; 2, s)}), we have: 
\begin{align*}
G&=\langle a, c \mid\  
a^{3}(    a^{-1}   c  )^s a^{2}=
 c^{2} (  a^{-1}   c  )^s c  \rangle \\
&=\langle a, c, b \mid\  
a^{3}b^s a^{2}=
 c^{2} b^s c, b=a^{-1}c  \rangle \\
 &=\langle a,  b \mid\  
a^{2}b^s a^{2}=bab^{s+1}ab  \rangle \\
&=\langle a, b, x \mid a^{2}b^sa^{2}=ba b^{s+1} ab, x=ab^{s-1}\rangle \\
&=\langle b, x \mid (xb^{-(s-1)})^2 b^s (xb^{-(s-1)})^2=bxb^{-(s-1)}b^{s+1}xb^{-(s-1)}b\rangle\\
&=\langle b, x \mid xb^{-(s-1)}xbxb^{-(s-1)}x=bxb^2xb \rangle\\
&=\langle b, x, y \mid xb^{-(s-1)}xbxb^{-(s-1)}x=bxb^2xb, y=bxb \rangle\\
&=\langle b,x,y \mid xb^{-(s-1)}xbxb^{-(s-1)}x=y^2, y=bxb \rangle\\
&=\langle b, y \mid y^2=(b^{-1}yb^{-1}) b^{-(s-1)} (b^{-1}y b^{-1}) b  ( b^{-1}yb^{-1} )b^{-(s-1)}  (b^{-1}yb^{-1}) \rangle\\
&=\langle b, y \mid y^2=b^{-1}yb^{-(s+1)} y b^{-1}y b^{-(s+1)}y b^{-1}\rangle.
\end{align*}

Let $w(b^{-1},y)=b^{-1}yb^{-(s+1)} y b^{-1}y b^{-(s+1)}y b^{-1}$.
Then the relation says $y^2=w(b^{-1},y)$.
Consider the commutator $[y,w(b^{-1},y)]$, which is the identity.
However, Proposition~\ref{w_x} shows that this commutator can be decomposed into a product of conjugates of 
$[y, b^{-1}]$ only.
Here $[y, b^{-1}]\ne 1$, 
for otherwise $K(5,3; 2,s)$ would be trivial.
However, $K(5,3;2,s)$ is a fibered knot of genus $s+4$ as in the proof of Theorem \ref{thm:pm}.
Thus $[y,w(b^{-1},y)]$ is a generalized torsion element in $G$.
\end{proof}

We should remark that the twisted torus knot $K(5,3;2,s)$ is the $(-2,3,2s+5)$-pretzel knot.
In particular, 
$(-2,3,7)$-pretzel knot admits a generalized torsion element in its knot group. 
Also, as mentioned in the proof of Theorem \ref{thm:532s},
this knot has genus $s+4$.
Furthermore, this is hyperbolic, except the torus knot $K(5,3;2,0)$ \cite{K}.
Thus these pretzel knots give another examples of Corollary \ref{g-torsion_large_genus}
realizing arbitrarily high genus.

Also, Theorem \ref{thm:532s} implies the following.

\begin{cor:pretzel}
The knot group of the pretzel knot of type $(-2,3,2n+5)$ is not bi-orderable for $n\ge 0$. 
\end{cor:pretzel}


Since this pretzel knot is fibered,  if the Alexander polynomial has no positive real root,
then the knot group is not bi-orderable \cite{CR}.
In fact, the Alexander polynomial of the pretzel knot of type $(-2,3,2n+5)$ is
\[
\Delta(t)=t^{2n+8}-t^{2n+7}+(t^{2n+5}-t^{2n+4}+t^{2n+3}-t^{2n+2}+\dots-t^4+t^3)-t+1
\]
as given in \cite{H}.
It is easy to see that there is no positive real root 
(consider the cases $t\ge 1$ and $0\le t<1$ ).
Hence, this gives another proof of Corollary \ref{pretzelknot}. 
(Note that the absence of bi-ordering does not imply the existence of generalized torsion elements.)

Recently,  Johnson \cite{J}  examines the bi-orderability for genus one pretzel knots.


\begin{remark}
\label{s>0}
It should be interesting to compare twist families given in Theorem~\ref{twist_T} and 
that given in Theorem~\ref{K(5,3; 2,s)}. 
In the former we may twist $K_q$ about $c_n$ in both positive and negative direction to obtain knots with generalized torsion. 
On the contrary, 
the latter family forces us to perform only positive twisting for technical reason. 
We wonder if this condition is necessary in the latter.  
In other words, if $s < 0$, then does the knot obtained by $s$ twisting have bi-ordering? 
\end{remark}

\section*{Acknowledgements}
We would like to thank to Eiko Kin for helpful conversation.
The first named author has been partially supported by JSPS KAKENHI Grant Number 19K03502 and Joint Research Grant of Institute of Natural Sciences at Nihon University for 2021.
The second named author has been supported by JSPS KAKENHI Grant Number 20K03587.

\end{document}